\documentclass[a4paper,11pt]{article}
\usepackage[margin=3cm]{geometry}

\usepackage{graphicx}
\usepackage{verbatim,color}
\usepackage{amsthm,amsmath,amsfonts}
\usepackage{enumerate,url,subfigure}
\usepackage{xspace}
\usepackage{hyperref}
\usepackage{paralist}
\usepackage[affil-sl,]{authblk}
\usepackage[inline]{enumitem}
\usepackage{breqn}

\theoremstyle{plain}
\newtheorem{theorem}{Theorem} 
\newtheorem{lemma}{Lemma}

\newtheorem{corollary}{Corollary}
\newtheorem{proposition}{Proposition}

\begin{document}
\newenvironment{pf}{\emph{Proof.}~}{}

\newcommand{\down}[1]{\raisebox{-.4ex}{#1}} 
\newcommand{\plus}{\ding{59}}
\newcommand{\lf}{\texttt{leaf}}
\newcommand{\true}{\textit{true}\xspace}
\newcommand{\false}{\textit{false}\xspace}

\makeatletter
    \renewcommand*{\@fnsymbol}[1]{\ensuremath{\ifcase#1\or *\or \dagger\or \S\or
       \mathsection\or \mathparagraph\or \|\or **\or \dagger\dagger
       \or \ddagger\ddagger \else\@ctrerr\fi}}
\makeatother

\title{\LARGE\bf On the Planar Split Thickness of Graphs\thanks{A preliminary version
  of this work appeared at the 12th Latin American Theoretical Informatics 
  Symposium (LATIN'16)~\cite{ekk+-opstg-LATIN16}.}\vspace{.5em}} 

\setlength{\affilsep}{2em}
\author[1]{David~Eppstein}
\author[2]{Philipp~Kindermann}
\author[3]{Stephen~Kobourov}
\author[4]{Giuseppe~Liotta}
\author[5]{Anna~Lubiw}
\author[6]{Aude~Maignan}
\author[5]{Debajyoti~Mondal}
\author[5]{Hamideh~Vosoughpour}
\author[7]{Sue~Whitesides}
\author[8]{Stephen~Wismath}

\affil[1]{University of California, Irvine, USA
  \texttt{\url{eppstein@uci.edu}}}
\affil[2]{FernUniversit\"at in Hagen, Germany.
  \texttt{\url{philipp.kindermann@fernuni-hagen.de}}}
\affil[3]{University of Arizona, USA.
  \texttt{\url{kobourov@cs.arizona.edu}}}
\affil[4]{Universit\`a degli Studi di Perugia, Italy.
  \texttt{\url{giuseppe.liotta@unipg.it}}}
\affil[5]{University of Waterloo, Canada.
  \texttt{\mbox{\{\href{mailto:alubiw@uwaterloo.ca}{alubiw} | \href{mailto:dmondal@uwaterloo.ca}{dmondal} | \href{mailto:hvosough@uwaterloo.ca}{hvosough}\} @uwaterloo.ca}}}
\affil[6]{Universit. Grenoble Alpes, France.
  \texttt{\url{aude.maignan@univ-grenoble-alpes.fr}}}
\affil[7]{University of Victoria, Canada.
  \texttt{\url{sue@uvic.ca}}}
\affil[8]{University of Lethbridge, Canada.
  \texttt{\url{wismath@uleth.ca}}}
\date{}

\maketitle
\setcounter{footnote}{0}
\begin{abstract}
Motivated by applications in graph drawing and information visualization,
we examine the planar split thickness of a graph, that is, the smallest~$k$ 
such that the graph is $k$-splittable into a planar graph. 
A $k$-split operation substitutes
a vertex $v$ by at most $k$ new vertices such that each neighbor of $v$
is connected to at least one of the new vertices.

We first examine the planar split thickness of complete graphs, complete 
bipartite graphs, multipartite graphs, bounded degree graphs, and genus-1 graphs.
We then prove that it is NP-hard to recognize graphs that 
are $2$-splittable into a planar graph, and show that one can approximate the 
planar split thickness of a graph within a constant factor. If the treewidth 
is bounded, then we can even verify $k$-splittability  in linear time, for a 
constant~$k$. 
\end{abstract}

\section{Introduction}
Transforming one graph into another by repeatedly applying an operation such 
as vertex/edge  deletion, edge flip or vertex split  is a classic problem in 
graph theory~\cite{liebers2001}. In this paper, we examine graph 
transformations under the vertex split operation. Specifically, a 
\emph{$k$-split operation} at some vertex~$v$ inserts at most~$k$ new 
vertices $v_1,v_2,\ldots,v_k$ in the graph, then, for each neighbor~$w$ of~$v$,
adds at least one
edge~$(v_i,w)$ where~$i\in[1,k]$, and finally deletes~$v$ along with its incident edges.
We define a \emph{$k$-split} of graph $G$ as a graph~$G^k$ that is obtained by 
applying a $k$-split to each vertex of~$G$ at most once.
We say that $G$ is \emph{$k$-splittable} into $G^k$.  
If $\mathcal{G}$ is a graph property, we say that $G$ is \emph{$k$-splittable} into a $\mathcal G$ graph  (or ``\emph{$k$-splittable} into $\mathcal{G}$'') if there is a $k$-split of $G$ that has property $\mathcal{G}$.
We introduce the \emph{$\mathcal{G}$ split thickness} of a graph~$G$ as the 
minimum integer~$k$ such that~$G$ is $k$-splittable into a 
$\mathcal{G}$ graph.

Graph transformation via vertex splits 
is important
in graph drawing and information visualization~\cite{GansnerHK10,RicheD10}. For example, suppose that we want to visualize the subset relation among a collection~$S$ of~$n$ sets. Construct an $n$-vertex graph~$G$ 
with a vertex for each set and an edge when one set is a subset of another.
A planar drawing of this graph gives a nice visualization of the subset relation.
Since the graph is not necessarily planar, a natural approach is to split~$G$ into a planar graph and then visualize the resulting graph, as illustrated in Figure~\ref{fig:motivation}(a). 
Let's now consider another interesting scenario where we want to visualize a graph~$G$ of a social network, see Figure~\ref{fig:motivation}(b). First, group the vertices of the graph into clusters by running a clustering algorithm. Now, consider the cluster graph: every cluster is a node and there is an edge between two cluster-nodes if there exists a pair of vertices in the corresponding clusters that are connected by an edge. In general, the cluster graph is non-planar, but we would like to draw the clusters in the plane. Thus, we may need to split a cluster into two or more sub-clusters. The resulting ``cluster map'' will be confusing if clusters are broken into too many disjoint pieces, which 
leads to 
the question of minimizing the planar split thickness.

\begin{figure}[t] \centering
\includegraphics[width=\textwidth,page=1]{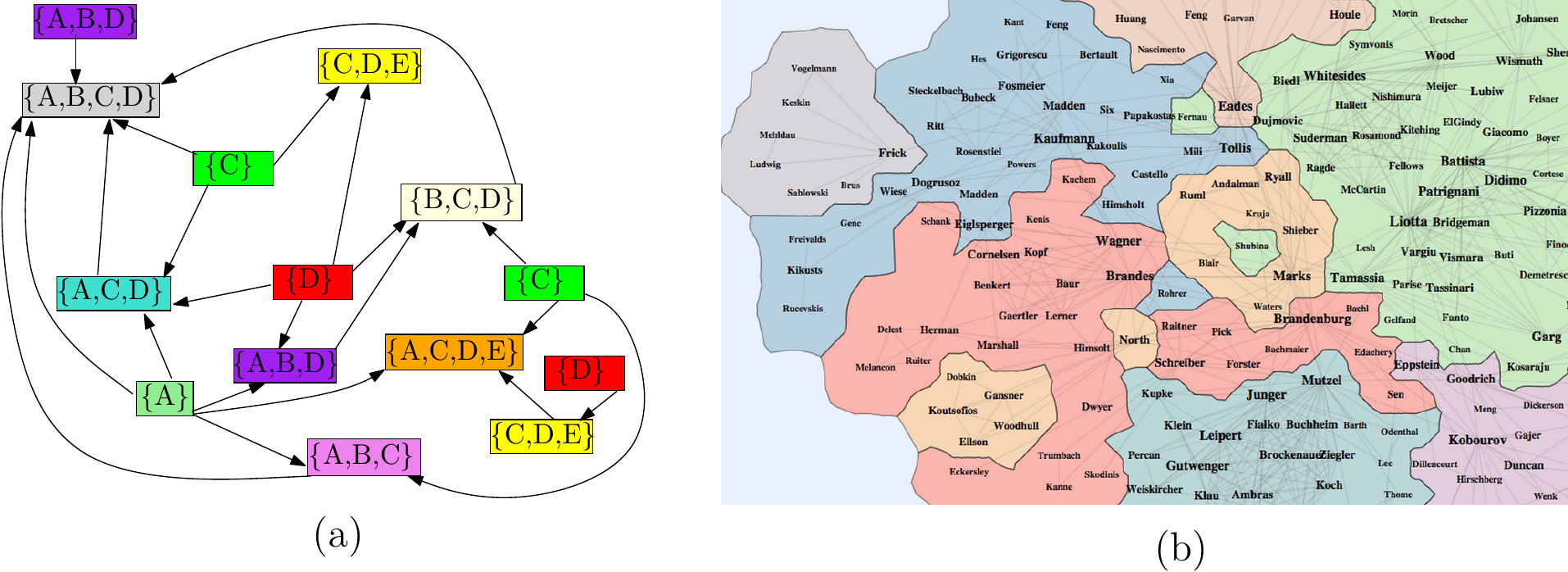} 
\caption{(a) A $2$-split visualization of subset relations among 10 sets.
 (b) Visualization of a social network. Note that the orange cluster has 3 sub-clusters, and the red cluster has 2 sub-clusters.}
\label{fig:motivation}
\end{figure}

\subsection{Related Work}
The problem of determining the planar split thickness of a graph~$G$ is 
 related to the graph thickness~\cite{Beineke65}, 
empire-map~\cite{hutchinson93}, $k$-splitting~\cite{liebers2001} and planar emulator~\cite{Chimani2013} problems. 
The \emph{thickness} of a graph~$G$ is the minimum integer~$t$ such that~$G$
admits an edge-partition into~$t$ planar subgraphs. One can assume that these 
planar subgraphs are obtained by applying a $t$-split operation at each 
vertex. Hence, thickness is an upper bound on the planar split thickness, 
e.g., the thickness and thus the planar split thickness  of graphs with 
treewidth~$\rho$ and maximum-degree-4 is at most
$\lceil \rho/2\rceil$~\cite{DujmovicW07} and 2~\cite{DuncanEK04}, 
respectively.  Analogously, the planar split thickness of a graph is bounded
by its \emph{arboricity}, that is, the minimum number of forests into which its 
edges can be partitioned. We will later show that both parameters also provide
an asymptotic lower bound on the planar split thickness.

A \emph{$k$-pire map} is a $k$-split planar graph, where an \emph{empire} consists of the copies of one original vertex (so each \emph{empire} consists of at most~$k$ vertices). 
 In 1890, Heawood~\cite{H90} proved that every~$12$ mutually adjacent empires can be drawn as a $2$-pire map where each empire has size exactly 2.
 Later, Ringel and Jackson~\cite{jr-shepp-JRAM84} showed that for every integer~$k\ge 2$ a set of~$6k$ mutually adjacent empires can be drawn as a $k$-pire map. This implies an upper bound of~$\lceil n/6 \rceil$ on the planar split thickness of a complete graph on~$n$ vertices.

A rich body of literature considers the planarization of non-planar graphs via
\emph{vertex splits}~\cite{FariaFM98,HartsfieldJR85,liebers2001,NetoSXSFF02}. 
Here a \emph{vertex split} is one of our 2-split operations. 
These results focus on minimizing the \emph{splitting number}, i.e., the 
total number of vertex splits to obtain a planar graph. Tight bounds on the splitting number are known 
for complete graphs~\cite{HartsfieldJR85} and complete bipartite 
graphs~\cite{JR84,JR85b}, but for general graphs, the problem of determining the 
splitting number of a graph is NP-hard~\cite{FariaFM98}. Note that upper 
bounding the splitting number does not necessarily guarantee any good upper 
bound on the planar split thickness, e.g., see Section~\ref{sec:cbp}.

Knauer and Ueckerdt~\cite{ku-3wcag-arXiv12} studied the \emph{folded covering number} 
 which is equivalent to our problem and stated several results for splitting 
graphs into a star forest, a caterpillar forest, or an interval graph. They showed that 
planar graphs are 4-splittable into a star forest, and planar bipartite graphs 
as well as outerplanar graphs are 3-splittable into a star forest. 
It follows from Scheinerman and West~\cite{sw-tinpg-JCTB83}
that planar graphs are 3-splittable into an interval graph and 4-splittable into
a caterpillar forest, while outerplanar graphs are 2-splittable into an interval
graph.

A \emph{planar emulator} is a  $k$-split planar graph with the additional property that for every original edge $(u,v)$ and every copy $v_i$ of vertex $v$  the $k$-split contains an edge $(v_i, u_j)$ for some copy $u_j$ of $u$.  (Planar split thickness requires this only for one copy of $v_i$.)
Not every graph has a planar emulator and it is an open problem to characterize those that do~\cite{Chimani2013}.   
It has been shown that the complete bipartite graph~$K_{3,5}$ and the graph $K_7-C_4$,
and thus the complete graph~$K_7$, have no finite
planar emulator~\cite{f-egg-phd85,h-ctgt-CM91}, although they are 2-splittable
(see Theorem~\ref{thm:complete} and Corollary~\ref{cor:bipartite}).
A \emph{planar cover} has the even stronger property that the copy $u_j$ is unique.  Negami conjectured in 1988~\cite{Negami} that a graph has a (finite) planar cover if and only if it embeds in the  projective plane.

\subsection{Our Contribution} 
In this paper, we examine the planar split thickness for 
non-planar graphs. 
Initially, we focus on splitting various graph classes into planar graphs,
namely complete graphs, complete bipartite graphs, 
graphs of bounded maximum degree, and graphs of 
(non\mbox{-}\nobreak)\nobreak\hspace{0pt}orientable
genus~1. 
We then prove that it is NP-hard to recognize graphs that 
are $2$-splittable into a planar graph, while we describe a technique for 
approximating the planar split thickness within a constant factor. 
Finally, for bounded treewidth graphs,  we present a technique to verify planar 
 $k$-splittability in linear time, for any constant $k$.

 Because our results are for planar $k$-splittability, we will drop the word ``planar'', and use ``$k$-splittable'' and ``$k$-split graph'' to mean ``planar
$k$-splittable'' and ``planar $k$-split graph'', respectively.
 The rest of the paper is organized as follows. In Section~\ref{sec:bp} we present results about complete and complete bipartite graphs. 
 In Section~\ref{sec:np} we prove the NP-hardness of recognizing 2-splittable graphs.  In Section~\ref{sec:coping} we present the results about approximation algorithms and fixed parameter tractability. Finally, in Section~\ref{sec:con} we summarize the results in the paper and consider directions for future research.
  
\section{Planar Split Thickness of Various Graph Classes} 
\label{sec:bp}

In this section, we focus on the planar split thickness of complete graphs,
complete bipartite graphs, graphs of bounded
maximum degree, and graphs of (non-)orientable genus~1.

\subsection{Complete Graphs}

Let $f(G)$ be the planar split thickness of the graph~$G$. Recall that  Ringel and Jackson~\cite{jr-shepp-JRAM84} showed that 
 $f(K_n)\le \lceil n/6 \rceil$ for every~$n\ge 12$. Since an $(n/6)$-split of an $n$-vertex graph contains at most~$n^2/2-6$ edges, 
 and the largest complete graph with at most~$n^2/2-6$ edges is~$K_n$, this bound is tight. Besides, for every~$n<12$,
 it is straightforward to construct a $2$-split graph of~$K_{n}$ by deleting~$2(12-n)$ vertices from the $2$-split graph of~$K_{12}$.
 Hence, we obtain the following theorem. 

\begin{theorem}[Ringel and Jackson~\cite{jr-shepp-JRAM84}]\label{thm:complete}
  If $n\le4$, then $f(K_n)=1$, and if $5\le n\le12$, then $f(K_n)=2$. 
  Otherwise, $f(K_n) = \lceil n/6 \rceil$.
\end{theorem}

Let $K^2_{12}$ be any $2$-split graph of $K_{12}$. Then,~$K^2_{12}$ has a particular useful property, as stated in the following lemma.  

\begin{lemma}\label{non-adj}
Any planar embedding~$\Gamma$ of~$K^2_{12}$ is a triangulation, where each vertex of~$K_{12}$ is split exactly once
  and no two vertices that correspond to the same vertex in~$K_{12}$ are incident to the same face. 
\end{lemma}
\begin{proof}
$K_{12}$ has $66$ edges. The $2$-split operation  
produces a graph with at most twice the number of vertices and at least the original number of edges,
 so any (planar) graph $K^2_{12}$ has 24 vertices and $66$ edges, 
  since that is the smallest number of vertices for a 66-edge planar graph by Euler's formula.
  Further, 66 edges is the largest number of edges for a 24-vertex planar graph by Euler's formula.
  Therefore, $K^2_{12}$ must be maximally planar, with all faces triangles.
  Two copies of the same vertex cannot be adjacent, so they cannot
  lie on the same boundary of a triangle face.
\end{proof}

Let~$H$ be the graph consisting of~2 copies of~$K_{12}$ attached at a common vertex~$v$. Then,~$H$ provides an example of a graph that is not 2-splittable even though its edge count does not preclude its possibility of being 2-splittable.

\begin{lemma}
  The graph $H$ is not $2$-splittable.
\end{lemma}
\begin{proof}  
 Consider a 2-split graph~$H'$ of one copy of~$K_{12}$. By Lemma~\ref{non-adj}, 
  the vertices~$v_1$ and~$v_2$ in~$H'$ that correspond to the same vertex 
  in~$K_{12}$ are not incident to the same face.  Since~$v$ can be split only 
  once,  the 2-split graph~$H''$ of the other copy of~$K_{12}$ must lie inside 
  some face that is incident to either~$v_1$ or~$v_2$.  Without loss of 
  generality, assume that it is inside some face incident to~$v_1$. Note that 
  both~$H'$ and~$H''$ need a copy of~$v$ in some face which is not incident 
  to~$v_1$. Since both~$H'$ and~$H''$ are triangulations, this would introduce a 
  crossing in any 2-split graph of~$H$.
\end{proof}

\subsection{Complete Bipartite Graphs} 
\label{sec:cbp}

Hartsfield et al.~\cite{HartsfieldJR85} showed that the splitting number of~$K_{m,n}$, where $m,n\ge 2$, is exactly~$\lceil (m-2)(n-2)/2\rceil$. However, their construction does not guarantee tight bounds on the splitting thickness of complete bipartite graphs. For example, if~$m$ is an even number, then their construction does not duplicate any vertex of the set~$A$ with~$m$ vertices, but uses $n+(m/2-1)(n-2)$ vertices to represent the set~$B$ of~$n$ vertices. Therefore, at least one vertex in the set~$B$ is duplicated at least $(n+(m/2-1)(n-2))/n= m/2-m/n+2/n\ge3$ times, for~$m\ge 6$ and~$n\ge 5$. On the other hand, we show that~$K_{m,n}$ is 2-splittable in some of these cases, as stated in the following theorem.  

\begin{figure}[pt] \centering
\subfigure[]{\includegraphics[width=.75\columnwidth]{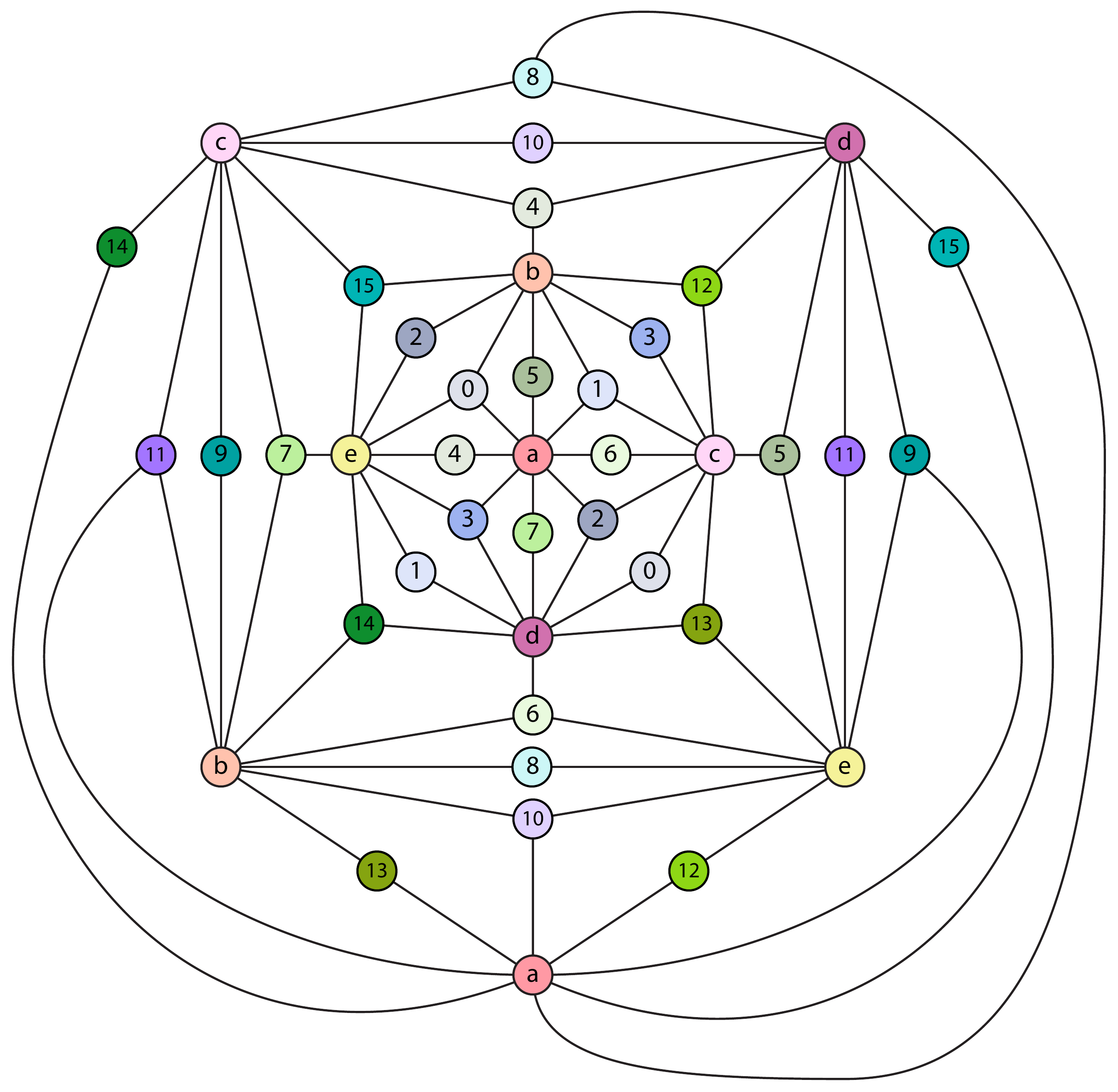}}\\
\bigskip
\subfigure[]{\includegraphics[width=.45\columnwidth]{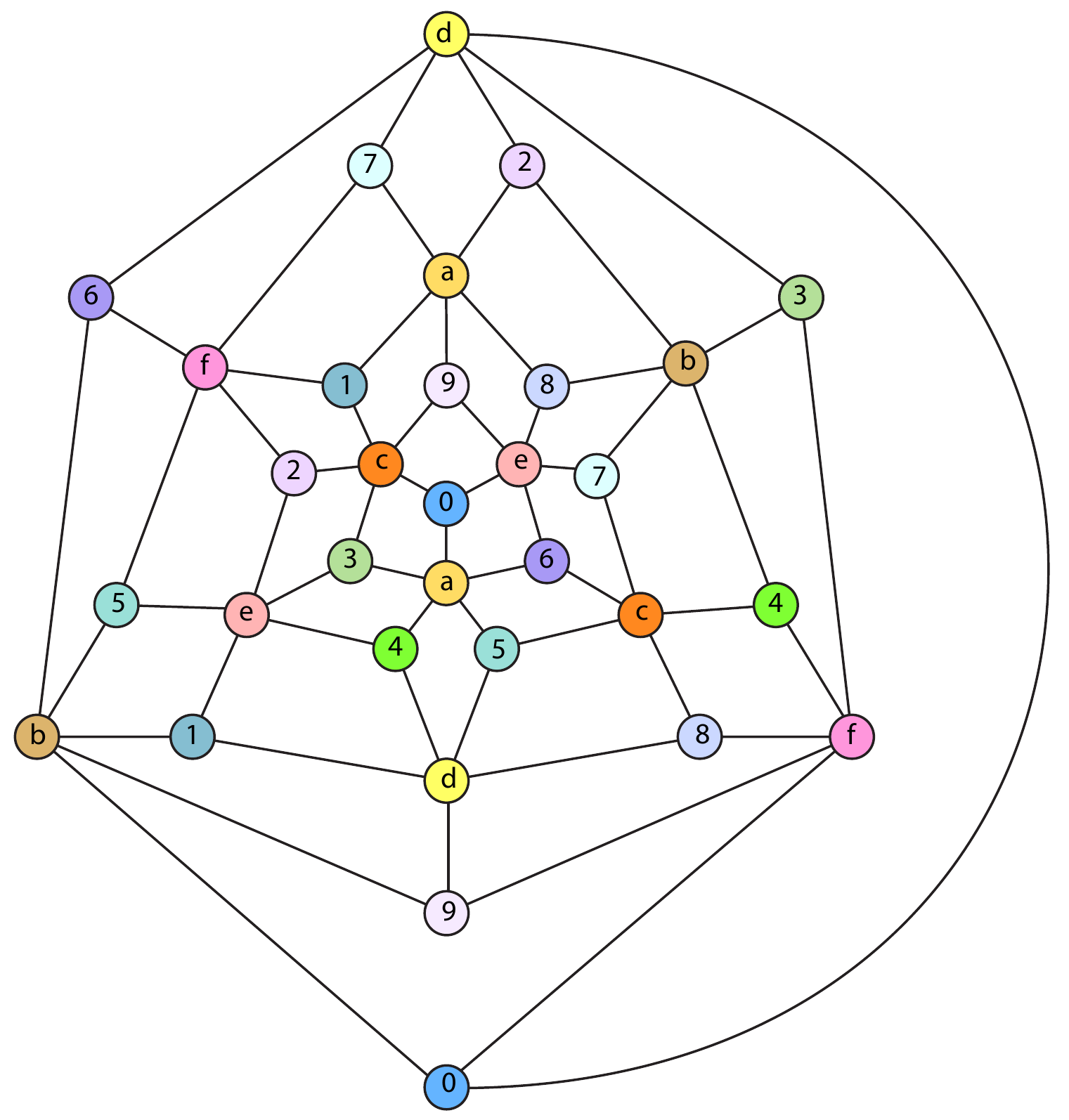}}\hfill
\subfigure[]{\includegraphics[width=.45\columnwidth]{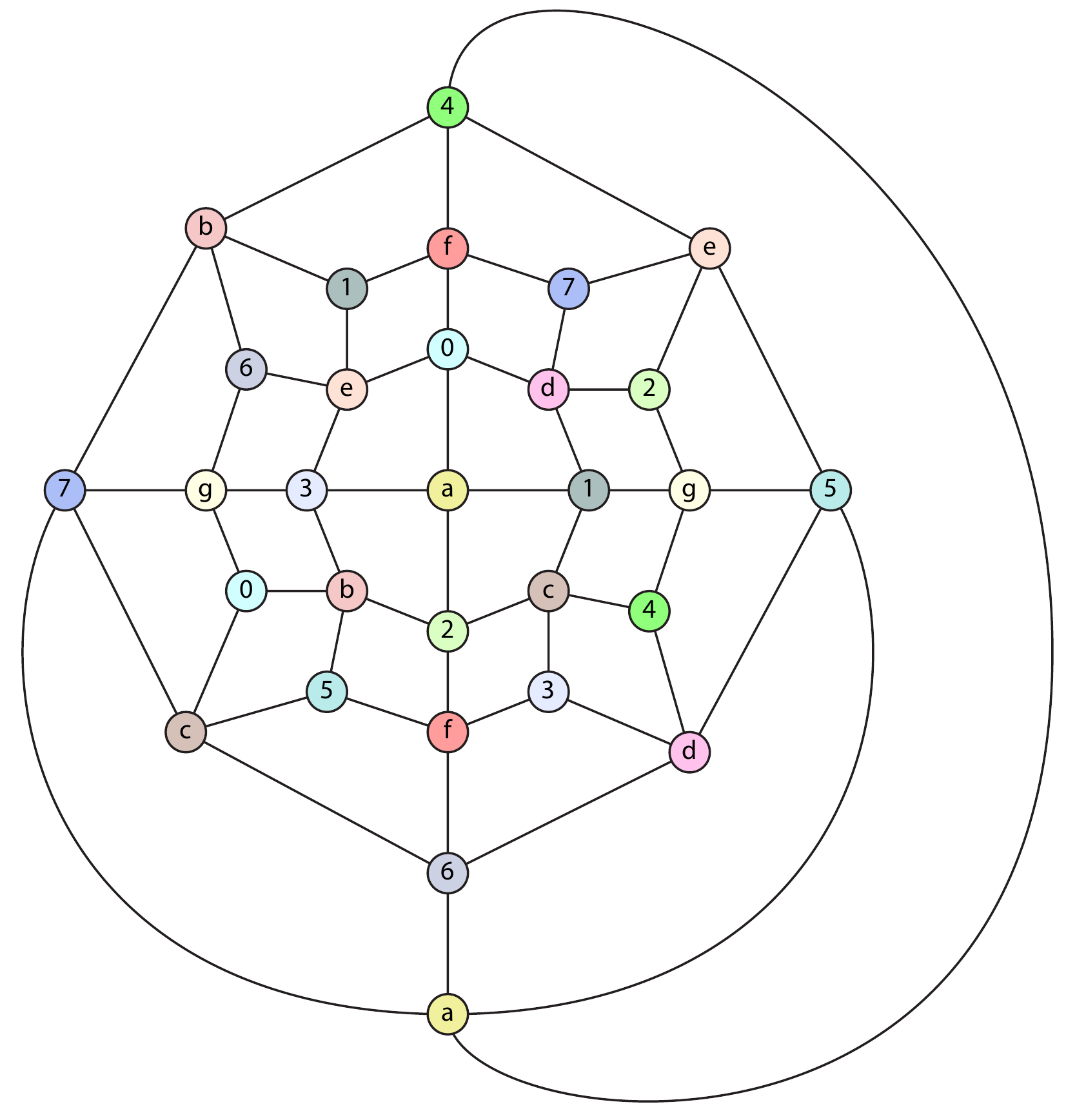}}
\caption{The $2$-split graphs of (a) $K_{5,16}$, (b) $K_{6,10}$, and (c) $K_{7,8}$. 
  }
\label{fig:bp}
\end{figure}

\begin{lemma}\label{lem:2bipartite}
The graphs~$K_{5,16}$, $K_{6,10}$, and $K_{7,8}$ are 2-splittable, and their
2-split graphs are quadrangulations, 
which implies that  for complete bipartite graphs $K_{m,n}$, where $m=5,6,7$, those are the largest graphs with planar split thickness 2. 
\end{lemma}
\begin{proof}
The sufficiency can be observed from the $2$-split construction of $K_{5,16}$, $K_{6,10}$, and $K_{7,8}$, as shown in Figure~\ref{fig:bp}.  A planar bipartite graph can have at most  $2n-4$ edges~\cite{HartsfieldJR85}. Since the graphs $K_{5,16},K_{6,10}$ and $K_{7,8}$ contain exactly $4(m+n)-4$ edges, their 2-split graphs are quadrangulations, which in turn implies that the result is tight:
any vertex that we add to one of the partitions has degree at least~5, so we
have to add a vertex of degree at least~3 to the 2-split graph, which cannot
be done planarly since each quadrangular face contains only~2 vertices of
each partition.
\end{proof}

With this Lemma, we can fully characterize the 2-splittable complete
bipartite graphs.

\begin{theorem}\label{thm:bipartite}
  Any complete bipartite graph $K_{m,n}$ is 2-splittable if and only if $nm\le 4(n+m)-4$.
\end{theorem}
\begin{proof}
  Without loss of generality, assume that $m\le n$. 
  
  If $m\le4$, then $4n\le 4(4+n)-4=4n+12$ is always satisfied.
  
  If $m=5$, then $5n\le 4(5+n)-4=4n+16\Leftrightarrow n\le 16$.
  
  If $m=6$, then $6n\le 4(6+n)-4=4n+20\Leftrightarrow n\le 10$.
  
  If $m=7$, then $7n\le 4(7+n)-4=4n+24\Leftrightarrow n\le 8$.
  
  If $m\ge 8$, then there is no $n>m$ that satisfies the inequality.
  
  Hence, the inequality is fulfilled exactly for the complete bipartite
  graphs that are a subgraph of $K_{4,n}$, $K_{5,16}$,
  $K_{6,10}$, or~$K_{7,8}$. The latter three graphs  are 2-splittable by 
  Lemma~\ref{lem:2bipartite}; for~$K_{4,n}$, we
  simply partition the graph into two copies of $K_{2,n}$, which are already 
  planar. Hence, the inequality is sufficient. On the other hand,
  Lemma~\ref{lem:2bipartite} also shows these four graphs are the largest graphs
  with planar split thickness 2 for $m\le 7$. In the proof of Lemma~\ref{lem:2bipartite},
  we showed that we cannot add a vertex to any of the bipartitions of~$K_{7,8}$; hence, there
  is no 2-splittable complete bipartite graph for $8\le m\le n$.
  This completes the proof.
  \end{proof}

This theorem can also be stated as follows.

\begin{corollary}\label{cor:bipartite}
  Any complete bipartite graph $K_{m,n}$ is 2-splittable if and only if
  it is a subgraph of $K_{4,n}$, $K_{5,16}$, $K_{6,10}$, or $K_{7,8}$.
\end{corollary}

In the following, we give some necessary conditions for $k$-splittable complete 
bipartite graphs based on the edge count argument.
Note that any  $k$-split graph $K_{m,n}^k$ of~$K_{m,n}$ must be a planar bipartite graph. Therefore,
if  $p$ and $q$ are the number of vertices and edges in $K_{m,n}^k$, respectively, then the inequality $q \le 2p-4$ holds.  
Consider a complete, $d$-vertex bipartite graph $K_{m,d-m}$ (with $m\le d/2$) that is $k$-splittable. 
The number of edges in this graph is $m \times (d-m)$.
  Since any $k$-split graph of $K_{m,d-m}$ can have at most~$kd$ vertices,  we have 
\begin{equation} m(d-m) \leq 2kd-4 \Leftrightarrow m^2-md+2kd-4 \geq 0 \label{eq2}
\end{equation} 
With Equation~\eqref{eq2}, we can prove the following propositions.

\begin{proposition} \label{prop:bp}
If $m\ge d/2$, $d \geq 4k+4 \sqrt{k^2-1}$ and  $m >\frac{ d-\sqrt{d^2-8kd+16}}{2}$, then $K_{m,d-m}$ is not $k$-splittable.
\end{proposition}
\begin{proof}

 The factorization of the Equation~\eqref{eq2} gives
\[ m^2-md+2kd-4=\left (m- \frac{d-\sqrt{d^2-8kd+16}}{2} \right) \left( m- \frac{d+\sqrt{d^2-8kd+16}}{2} \right),\]
and these constants are real numbers when $d \geq 4k+4 \sqrt{k^2-1}$. Therefore, for the equation to hold we need to have
$m \le  ({d-\sqrt{d^2-8kd+16}})/{2}$ or $m \ge ({d+\sqrt{d^2-8kd+16}})/{2}$.  
\end{proof}

\begin{proposition}\label{prop:bipart1}
If $k <  ({mn+4})/({2m+2n}) $, then $K_{m,n}$ is not $k$-splittable. 
\end{proposition}
\begin{pf}
 Equation~\eqref{eq2} for $d=m+n$ gives 
\begin{equation*}k \geq \frac{-m^2+md+4}{2d}=\frac{mn+4}{2m+2n}.\tag*{\qed}\end{equation*}
\end{pf}
 
\begin{proposition}
$K_{n,n}$ is not  $\lfloor n/4\rfloor$-splittable.
\end{proposition}
\begin{proof}
  To verify this, observe that 
  $K_{n,n}$ has $d=2n$ vertices and, for $k=\lfloor n/4\rfloor$, Equation~\eqref{eq2} gives  
  \begin{dmath*}
    n^2-n(2n)+2k(2n)-4=-n^2+4kn-4=-(n-2k-2 \sqrt{k^2-1})(n-2k+2 \sqrt{k^2-1})\geq 0.
  \end{dmath*}
  This constraint does not hold when $ n >2k+2 \sqrt{k^2-1}$. Furthermore, $n\ge 4\lfloor n/4\rfloor=4k>2k+2 \sqrt{k^2-1}$,
   which completes the proof.
\end{proof}

\begin{proposition}
$K_{2k+1, 4k^2+2k-3}$ is not $k$-splittable.
\end{proposition}
\begin{pf}
   To verify this, observe that if $m=2k+1$,  then by Equation~\eqref{eq2} we obtain 
  \begin{align*}&(2k+1)^2-(2k+1)d+2kd-4 \geq 0 \\
  \Leftrightarrow&~d \leq 4k^2+4k-3 \\
  \Leftrightarrow&~n \leq 4k^2+2k-4.\tag*{\qed}
  \end{align*}
\end{pf}

\begin{proposition}
$K_{2k, n} $ is $k$-splittable for every integer $n$.
\end{proposition}
\begin{proof}
  The proof for this claim is
  straightforward from the observation~$K_{2k,n}$ can be partitioned into~$k$
  copies of~$K_{2,n}$, which is planar. 
\end{proof}

\begin{table}[t]
  \centering
  \renewcommand{\arraystretch}{1.2}
  \begin{tabular}{c@{\quad}cc|cc@{\quad}cc|cc}
    \multicolumn{3}{c|}{$k=2$}        &   \multicolumn{4}{c|}{$k=3$}           & \multicolumn{2}{c}{$k\ge 4$}                            \\ \hline
    $m,14-m$ & $m \leq 4,n$  &&&  $m,22-m$ & $m \leq 6,n$ &&& $m, 4k+4 \sqrt{k^2-1} -m$  \\
    $5,n\leq 16$ & $6,n\leq 10$ &&& $7,n\leq 38$ & $8,n\leq 22$ &&&      $m \leq 2k,n$              \\
    $7,n\leq 8$ &               &&& $9,n\leq 16$ & $10,n\leq 14$    &&&   $m>2k ,n \leq \frac{2km-4}{m-2k}$                            \\
					&                  &&& $11,n\leq 12$ &        &&&                               
  \end{tabular}
  \medskip
  \caption{All complete bipartite graphs that fulfill the edge requirements of Equation~\eqref{eq2}. 
    A table entry $m,n$ corresponds to the complete bipartite graph $K_{m,n}$.}
  \label{table:kpartite}
  \renewcommand{\arraystretch}{1}
\end{table}

Table~\ref{table:kpartite} summarizes the observations above by listing all 
complete bipartite graphs which satisfy the necessary conditions
provided above for different values of~$k$.

\subsection{Graphs with Maximum Degree $\Delta$}

Recall that the planar split thickness of a graph is bounded by its arboricity. 
By definition, any maximum-degree-$\Delta$ graph has degeneracy\footnote{A graph $G$ is $k$-degenerate if every subgraph of $G$ contains a vertex of degree at most $k$.} at most $\Delta$
and, thus, arboricity at most~$\Delta$. Hence, the planar split thickness of a 
maximum-degree-$\Delta$ graph is bounded by~$\Delta$.

Moreover, since every graph of maximum degree~2 is planar,
the planar split thickness of any graph with maximum degree~$\Delta$ is bounded by $\lceil \Delta/2\rceil$:
to every vertex~$v$ of degree~$d$, we apply a $\lceil d/2\rceil$-split and arbitrarily
assign at most two edges to each copy of~$v$. This gives a $\lceil \Delta/2\rceil$-split
into a graph of maximum degre~2.
 Therefore, the planar split thickness of a maximum-degree-5 graph is at most~3. The following theorem shows that this bound is tight. 

\begin{theorem}
For any nontrivial minor-closed property~$P$, there exists a graph~$G$ of maximum degree five whose~$P$ split thickness is at least~3.
\end{theorem}

\begin{proof}
This follows from a combination of the following observations:
\begin{compactenum}
\item\label{obs:log-girth} There exist arbitrarily large 5-regular graphs with girth~$\Omega(\log n)$~\cite{Mor-JCTB-94}.
\item Splitting a graph cannot decrease its girth.
\item\label{obs:clique-threshold} For every~$h$, the $K_h$-minor-free $n$-vertex graphs all have at most~$O(nh\sqrt{\log h})$ edges~\cite{Tho-JCTB-01}.
\item\label{obs:dense-minor} Every graph with~$n$ vertices, $m$ edges, and girth~$g$ has a minor with~$O(n/g)$ vertices and $m-n+O(n/g)$ edges~\cite{BorEppZhu-GD-14}.
\end{compactenum}
Thus, let~$h$ be large enough that~$K_h$ does not have property~$P$.
If~$G$ is a sufficiently large $n$-vertex 5-regular graph with logarithmic 
girth (Observation~\ref{obs:log-girth}), then any 2-split of~$G$ will 
have at most~$2n$ vertices and at least~$5n/2$ edges. By 
Observation~\ref{obs:dense-minor}, this 2-split will have a minor whose number 
of edges is larger by a logarithmic 
factor than its number of vertices, and for~$n$ sufficiently large this 
factor will be large enough to ensure that a~$K_h$ minor exists within the 
2-split of~$G$ (by Observation~\ref{obs:clique-threshold}). Thus,~$G$ cannot be 
2-split into a graph with property~$P$.
\end{proof}

\subsection{Graphs of (non-)orientable genus~1}

The splitting number has been studied for the projective 
plane~\cite{Hartsfield87} and on the torus~\cite{Hartsfield86}. Hence, it is 
natural to study split thickness on different surfaces. 

\begin{theorem}
  Any graph~$G$ of (non-)orientable genus~1 is 2-splittable.
\end{theorem}
\begin{proof}
  For graphs of orientable genus~1, that is, graphs that are embeddable on
  the torus, we draw the torus as a box with periodic boundary conditions.
  For the edges that cross the bottom boundary, we apply a split to the
  vertex whose edge part goes to the bottom boundary, and place its copy above 
  the top boundary. Then, we can simply reroute the edges that cross the left
  and right boundary around the whole drawing; see Figure~\ref{fig:genus}.
   
  For graphs of non-orientable genus~1, that is, graphs that are embeddable on the
  projective plane, it has been shown by Negami~\cite{n-eppeg-DM86} that 
  they have 2-fold planar cover which can be obtained by their preimage in
  the canonical double covering of the projective plane, which is a sphere.
  This implies that these graphs are 2-splittable.
\end{proof}

\section{NP-hardness}
\label{sec:np}

Faria et al.~\cite{FariaFM98} showed that determining the splitting number of a graph is NP-hard, even when the input is restricted to cubic graphs. Since cubic graphs are 2-splittable, their hardness proof does not readily imply the hardness of recognizing 2-splittable graphs. In this section, we show that it is indeed NP-hard to recognize graphs that are 2-splittable into a planar graph.

The reduction is from planar 3-SAT with a cycle through the clause vertices~\cite{KratochvilLN91}.
Specifically the input is an instance of 3-SAT with variables $X$ and clauses $C$ 
such that the following graph is planar: the vertex set is $X \cup C$; we add edge 
$(x,c)$ if variable $x$ appears in clause $c$; and we add a cycle through all the clause vertices. 
Kratochv{\'{\i}}l et al.~\cite{KratochvilLN91} showed that this version of 3-SAT (\textsc{Planar Cycle 3-SAT}) remains NP-complete.

\begin{figure}[t]
  \centering
  \includegraphics{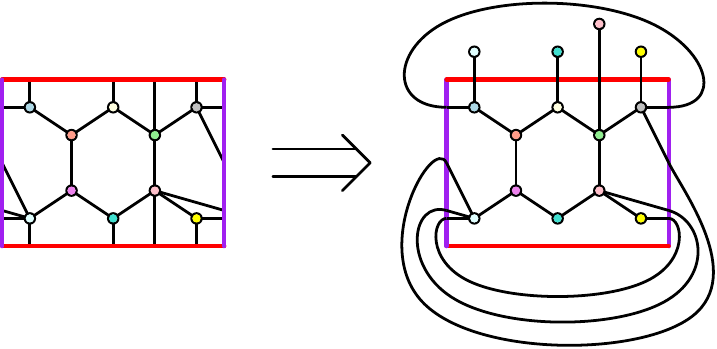}
  \caption{A drawing of a toroidal graph.}
  \label{fig:genus}
\end{figure}

For our construction, we will need to restrict the splitting options for some vertices.     
For a vertex $v$, \emph{attaching~$K_{12}$ to $v$} means inserting a new copy of $K_{12}$ into the graph and identifying~$v$ with a vertex of this~$K_{12}$. 
A vertex that has a $K_{12}$ attached will be called a ``K-vertex".

\begin{lemma}
If $C$ is a cycle of K-vertices then in any planar 2-split, the cycle $C$ appears intact, i.e.~for each edge of $C$ there is a copy of the edge in the 2-split such that the copies are joined in a cycle.  
\label{lemma:no-split}
\end{lemma}

\begin{proof}
  Let~$v$ be a vertex of cycle~$C$.  We will argue that the two edges 
  incident to~$v$ in~$C$ are incident to the same copy of~$v$ in the planar 
  2-split.  This implies that the cycle appears intact in the planar 2-split.

  Suppose the vertices of $C$ are $v=c_0,c_1, c_2, \ldots, c_t$ in that order, with 
  an edge $(v,c_t)$. As noted earlier in the paper, a planar 2-split of~$K_{12}$ 
  must split all vertices, and no two copies of a vertex share a face in the 
  planar 2-split.  Furthermore, any planar 2-split of~$K_{12}$ is connected.

  Let~$H_i$ be the induced planar 2-split of the~$K_{12}$ incident to~$c_i$. 
  Let~$v^1$ and~$v^2$ be the two copies of~$v$ in~$H_0$. Suppose that the copy
  of edge $(v,c_1)$ in the planar 2-split is incident to~$v^1$.  Our goal is to
  show that the copy of edge~$(v,c_t)$ in the planar 2-split is also incident 
  to~$v^1$. $H_1$ must lie in a face~$F$ of~$H_0$ that is incident to~$v^1$. 
  Since there is an edge $(c_1,c_2)$, $H_2$ must also lie in face~$F$ of~$H_0$.  
  Continuing in this way, we find that~$H_t$ must also lie in the face~$F$.  
  Therefore, the copy of the edge $(c_t,v)$ must be incident to~$v^1$ in the 
  planar 2-split.
\end{proof}

Note that the Lemma extends to any 2-connected subgraph of K-vertices.

Given an instance of \textsc{Planar Cycle 3-SAT}, we construct a graph as follows.   We will make a K-vertex~$c_j$ for each clause $c_j$, and join them in a cycle as given in the input instance.  By the Lemma above, this ``clause'' cycle will appear intact in any planar 2-split of the graph. 

Let $T$ be any other cycle of K-vertices, disjoint from the clause cycle.  $T$ will also appear intact in any planar 2-split, so we can identify the ``outside'' of the cycle $T$ as the side that contains the clause cycle.  The other side is the ``inside''.

For each variable~$v_i$, we create a vertex gadget as shown in Figures~\ref{fig:var-gadget-1}--\subref{fig:var-gadget-2} with six K-vertices: two special vertices~$v_i$ and~$\bar v_i$ and four other vertices forming  a ``variable cycle'' $v_i^1, v_i^2, v_i^3, v_i^4$ together with two paths $v_i^1, v_i, v_i^3$ and $v_i^2, \bar v_i, v_i^4$.  Observe that, in an embedding of any planar 2-split, the vertex gadget will appear intact, and exactly one of~$v_i$ and~$\bar v_i$ must lie inside the variable cycle and exactly one must lie outside the variable cycle.  Our intended correspondence is that the one that lies outside is the one that is set to \true.

\begin{figure}[tb]
\centering
\subfigure[]{\includegraphics[page=1]{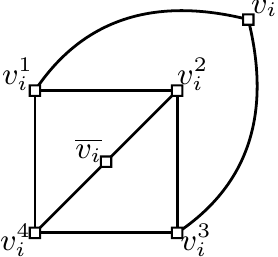}\label{fig:var-gadget-1}}
\hfil
\subfigure[]{\includegraphics[page=2]{variable-gadget}\label{fig:var-gadget-2}}
\hfil
\subfigure[]{\includegraphics{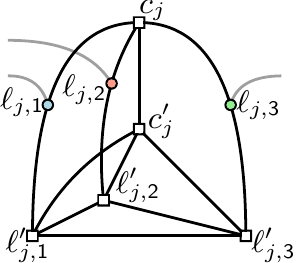}\label{fig:clause-gadget}}
\caption{\subref{fig:var-gadget-1} A variable gadget shown in the planar configuration corresponding to 
 $v_i=\true$ and \subref{fig:var-gadget-2} in the planar configuration corresponding to $v_i=\false$. \subref{fig:clause-gadget} A clause gadget---a~$K_5$ with added subdivision vertices $\ell_{j,1}, \ell_{j,2}, \ell_{j,3}$ corresponding to the literals in the clause.  The half-edges join the corresponding variable vertices.}
\label{fig:variable-clause}
\end{figure}

For each clause~$c_j$ with literals $\ell_{j,k}$, $k=1,2,3$, we create a~$K_5$ clause gadget, as shown in Figure~\ref{fig:clause-gadget}, with five K-vertices: two vertices  $c_j, c'_j$ and three vertices~$\ell'_{j,k}$.    Furthermore, we subdivide each edge $(c_j, \ell'_{j,k})$ by a vertex~$\ell_{j,k}$ that is {\it not} a K-vertex.   If literal~$\ell_{j,k}$ is~$v_i$, then we add an edge  $(v_i,\ell_{j,k})$ and if literal~$\ell_{j,k}$ is~$\bar v_i$, then we add an edge~$(\bar v_i,\ell_{j,k})$. 
  Figure~\ref{fig:example}  shows an example of the construction.
  
\begin{figure}[pt]
 \centering
\subfigure[]{\includegraphics[page=1]{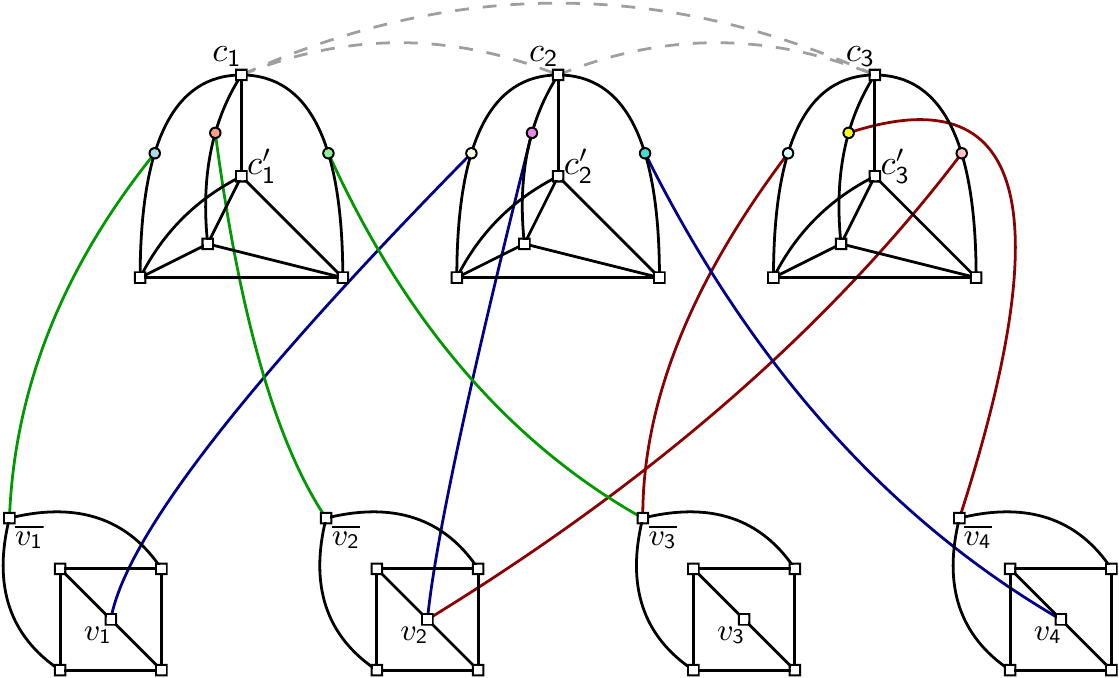}\label{fig:example:a}}
\\
\subfigure[]{\includegraphics[page=2]{example3SAT-new.pdf} \label{fig:example:b}}
\caption{\subref{fig:example:a} A graph that corresponds to the $3$-SAT instance $\phi = (\bar{v_1} \vee \bar{v_2}  
\vee \bar{v_3})\wedge (v_1 \vee v_2 \vee v_4) \wedge (v_2 \vee  \bar{v_3} \vee \bar{v_4})$. \subref{fig:example:b} A planarization of the graph in~\subref{fig:example:a} that satisfies $\phi$: $v_1 = \mbox{true}$, $v_2 = v_3 = v_4 = \mbox{false}$} 
\label{fig:example}
\end{figure} 

Note that the only non-K-vertices are the $\ell_{j,k}$'s, which have degree 3 and can be split in one of three ways as shown in Figures~\ref{fig:clause-splits-1}--\subref{fig:clause-splits-3}.  In each possibility, one edge incident to $\ell_{j,k}$ is ``split off'' from the other two.  If the edge to the variable gadget is split off from the other two, we call this the \emph{F-split}.

Observe that if, in the clause gadget for $c_j$, all three of $\ell_{j,1}, \ell_{j,2}, \ell_{j,3}$ use the F-split (or no split), then we effectively have edges from $c_j$ to each of $\ell'_{j,1}, \ell'_{j,2}, \ell'_{j,3}$, so the clause gadget is a $K_5$ which must remain intact after the 2-split and is not planar.  This means that in any planar 2-split of the clause gadget, at least one of $\ell_{j,1}, \ell_{j,2}, \ell_{j,3}$ must be split with a non-F-split.

\begin{lemma}\label{lem:satplanar}
  If the formula is satisfiable, then the graph has a planar 2-split.
\end{lemma}
\begin{proof}
For every literal $\ell_{j,k}$ that is set to \false, we do an F-split on the vertex~$\ell_{j,k}$.  For every literal $\ell_{j,k}$ that is set to \true, we split off the edge to $\ell'_{j,k}$; see Figure~\ref{fig:clause-splits-2}.  For any K-vertex $v$ incident to edges $E_v$ outside its $K_{12}$, we split all vertices of the $K_{12}$ as required for a planar 2-split of $K_{12}$ but we keep the edges of $E_v$ incident to the same copy of $v$, which we identify as the ``real'' $v$.

If variable $v_i$ is set to \true, we place (real) vertex $v_i$ outside the variable cycle and we place vertex $\bar v_i$ and its dangling edges inside the variable cycle.  If variable~$v_i$ is set to \false, we place vertex $\bar v_i$ outside the variable cycle and we place vertex~$v_i$ and its dangling edges inside the variable cycle.

Consider a clause $c_j$.  It has a true literal, say $\ell_{j,1}$.  We have split off the edge from $\ell_{j,1}$ to $\ell'_{j,1}$ which cuts one edge of the $K_5$ and permits a planar drawing of the clause gadget as shown in Figure~\ref{fig:clause-sat}, with $\ell'_{j,1}$ and its dangling edge inside the cycle $c', \ell'_{j,2}, \ell'_{j,3}$.

\begin{figure}[t]
\centering
\subfigure[]{\includegraphics[scale=.95,page=1]{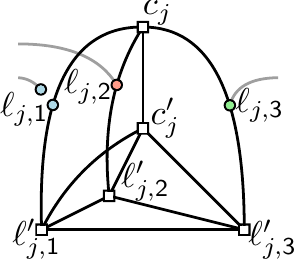}
\label{fig:clause-splits-1}}
\subfigure[]{\includegraphics[scale=.95,page=2]{clause-splits_new}
\label{fig:clause-splits-2}}
\subfigure[]{\includegraphics[scale=.95,page=3]{clause-splits_new}
\label{fig:clause-splits-3}}
\hfil
\subfigure[]{\includegraphics[scale=.95,page=4]{clause-splits_new}\label{fig:clause-sat}}
\caption{\subref{fig:clause-splits-1}--\subref{fig:clause-splits-3} The three ways of splitting $\ell_{j,1}$; \subref{fig:clause-splits-1} is the F-split. \subref{fig:clause-sat} A planar drawing of the clause gadget when  literal $\ell_{j,1}$ is set to \true and the split of vertex $\ell_{j,1}$ results in a dangling edge to $\ell'_{j,1}$. }  
\end{figure}

Because we started with an instance of planar 3-SAT with a cycle through the clause vertices, we know that the graph of clauses versus variables plus the clause cycle is planar.  We make a planar embedding of the split graph based on this, embedding the variable and clause gadgets as described above. 
The resulting embedding is planar.
\end{proof} 

\begin{lemma}\label{lem:planarsat}
  If the graph has a planar 2-split, then the formula is satisfiable.
\end{lemma}

\begin{proof}
Consider a planar embedding of a 2-split of the graph. 
As noted above, in each clause gadget, say $c_j$, at least one of the vertices $\ell_{j,k}$, $k=1,2,3$, must be split with a non-F-split.  Suppose that vertex $\ell_{j,k}$ is split with a non-F-split.  If literal $\ell_{j,k}$ is $v_i$ then we will set variable $v_i$ to \true; and if literal $\ell_{j,k}$ is $\bar v_i$  then we will set variable $v_i$ to \false.  We must show that this is a valid truth-value setting.
Suppose not.  Then, for some~$i$, vertex $v_i$ is joined to vertex $\ell_{j,k}$ that is split with a non-F-split, and vertex $\bar v_i$ is joined to vertex $\ell_{r,s}$ that is split with a non-F-split.
But then we essentially have an edge from $v_i$ to a vertex of the $c_j$ clause gadget and an edge from $\bar v_i$ to a vertex of the~$c_r$ clause gadget.  Because each clause gadget is a connected graph of K-vertices, and the clause gadgets are joined by the clause cycle, this gives a path of K-vertices from~$v_i$ to $\bar v_i$.  Then the 6 vertices of the variable gadget for $v_i$ form a subdivided $K_{3,3}$ of K-vertices.  This must remain intact under 2-splits and is non-planar.  This contradicts the assumption that there exists a planar 2-split of the graph. 
\end{proof}

With Lemmas~\ref{lem:satplanar} and~\ref{lem:planarsat}, we can proof the following theorem.

\begin{theorem}
  It is NP-hard to decide whether a graph has planar split thickness~2
  even when the maximum degree is restricted to 15. 
\end{theorem}
\begin{proof}
We first briefly review the NP-hardness proof~\cite{KratochvilLN91} for 
\textsc{Planar Cycle 3-SAT}. Given an instance $I$ of 3-SAT with each 
variable appearing in at most $\beta$ clauses, Kratochv\'{i}l et 
al.~\cite{KratochvilLN91} constructed a corresponding instance $I'$ of 
\textsc{Planar Cycle 3-SAT} such that $I$ admits a satisfying truth assignment 
if and only if $I'$ admits a satisfiable. Their construction ensures that each 
variable in $I'$ appears in at most $\max\{\beta,6\}$ clauses. 
Tovey~\cite{Tovey} showed that 3-SAT remains NP-complete even when every 
variable is restricted to appear in at most $3$-clauses, i.e., $\beta=3$.  
Thus, \textsc{Planar Cycle 3-SAT} remains NP-hard even when each variable 
appears in at most $6$ clauses. 

Consequently, the K-vertices we used in our 
reduction can be incident to at most 8 edges, where the maximum could be 
attained at some   K-vertex in the variable gadget. Since a K-vertex 
corresponds to a $K_{12}$, the maximum degree of the graph that we used in 
our hardness  reduction can be at most $19$. We can improve this to 15 by 
using $K_{7,8}$ as the K-vertex. Recall from the proof of 
Theorem~\ref{lem:2bipartite} that the number of edges in $K_{7,8}$ is 56, which 
is exactly the number of edges in a maximal planar bipartite graph of 30 
vertices. Hence, every vertex in $K_{7,8}$ must be split and no two copies of
the same vertex are incident to the same face, which are exactly the conditions we
need for our hardness reduction.
\end{proof}

\section{Approximation and Fixed Parameter Tractability}
\label{sec:coping}

In this section, we prove that the arboricity (respectively,  pseudoarboricity) of $k$-splittable graphs is bounded by~$3k+1$ (respectively, $3k$), and that testing $k$-splittability is fixed-parameter tractable in the treewidth of the given graph.

\subsection{Approximating Split Thickness}
The \emph{arboricity $a(G)$} of a graph~$G$ is the minimum integer such that $G$ admits a decomposition
 into $a(G)$ forests. By definition, the planar split thickness of a graph is bounded by its arboricity.
 We now show that  the arboricity of a $k$-splittable graph approximates its planar split thickness within a constant factor.

 Let $G$ be a $k$-splittable graph with $n$ vertices and let $G^k$ be a $k$-split graph of $G$. 
  Since $G^k$ is planar, it has at most $3kn-6$ edges. Therefore, the number of edges in~$G$
  is also at most $(3k+1)(n-1)$: for $n$ at most $6k$, this follows simply 
  from the fact that any $n$-vertex graph can have at most $n(n-1)/2$ edges, and for larger~$n$
  this modified expression is bigger than $3kn-6$. But Nash-Williams~\cite{nash64}
  showed that the arboricity of a graph is at most $a$ if and only if every $n$-vertex subgraph 
  has at most~$a(n-1)$ edges. Using this characterization and the bound on the number 
  of edges, the arboricity is at most $3k+1$.

A forest is called a \emph{pseudoforest} if it contains at most one cycle per 
connected component. The \emph{pseudoarboricity} $p(G)$ of a graph $G$ is the minimum 
integer such that $G$ admits  a decomposition into $p(G)$ pseudoforests. The 
pseudoarboricity of a graph is at most~$p$ if and only if every $n$-vertex 
subgraph has at most~$p\cdot n$ edges~\cite{JCP}. Since a 
$k$-splittable graph with $n$ vertices may have at most $3kn-6$ edges, the 
pseudoarboricity of such a  graph is at most $3k$.  

Note that the thickness of a graph is bounded by its pseudoarboricity, and thus also approximates the planar split thickness within factor~$3$. 
Furthermore, we note that arboricity and pseudo-arboricity can be computed in 
polynomial time~\cite{Gabow1992} so this gives a polynomial time approximation 
algorithm for split thickness.

\begin{theorem}
The arboricity (respectively, pseudoarboricity) of a $k$-splittable graph is bounded by $3k+1$ (respectively, $3k$), and therefore approximates its planar split thickness within factor~$3+1/k$ (respectively, $3$).
\end{theorem}

\subsection{Fixed-Parameter Tractability}
Although $k$-splittability is NP-complete, we show in this section that it is solvable in polynomial time for graphs of bounded treewidth. The result applies not only to planarity, but to 
many other graph properties.

\begin{theorem}
Let $P$ be a graph property, such as planarity, that can be tested in monadic second-order graph logic, and let $k$ and $w$ be fixed constants. Then it is possible to test in linear time whether a graph of treewidth at most~$w$ is $k$-splittable into~$P$ in linear time.
\end{theorem}

\begin{proof}
We use Courcelle's theorem~\cite{Cou-IC-90}, according to which any monadic 
second-order property can be tested in linear time for bounded-treewidth graphs.
We modify the formula for $P$ into a formula for the graphs 
$k$-splittable into~$P$.

To do so, we need to be able to distinguish the two endpoints of each edge of 
our given graph $G$ within the modified formula by building a depth-first search 
tree. To this end, we wrap the formula 
in existential quantifiers for an edge set $T$ and a vertex $r$, and we form 
the conjunction of the formula with the conditions that every partition of 
the vertices into two subsets is crossed by an edge, that every nonempty 
vertex subset includes at least one vertex with at most one neighbor in the 
subset, and that, for every edge $e$ that is not part of $T$, there is a path 
in $T$ starting from~$r$ whose vertices include the endpoints of $e$. These 
conditions ensure that $T$ is a depth-first search tree of the given graph, 
in which the two endpoints of each edge of the graph are related to each 
other as ancestor and descendant; we can orient each edge from its ancestor 
to its descendant~\cite{Cou-DCFM-96}.

With this orientation in hand, we wrap the formula in another set of existential quantifiers, asking for $k^2$ edge sets, and we add conditions to the formula ensuring that these sets form a partition of the edges of the given graph. If we number the split copies of each vertex in a $k$-splitting of the given graph from $1$ to $k$, then these $k^2$ edge sets determine, for each input edge, which copy of its ancestral endpoint and which copy of its descendant endpoint are connected in the graph resulting from the splitting.

Given these preliminary modifications, it is straightforward but tedious to modify the formula for $P$ itself so that it applies to the graph whose splitting is described by the above variables rather than to the input graph. To do so, we need only replace every vertex set variable by $k$ such variables (one for each copy of each vertex), expand the formula into a disjunction or conjunction of $k$ copies of the formula for each individual vertex variable that it contains, and modify the predicates for vertex-edge incidence within the formula to take account of these multiple copies.
\end{proof}

\section{Conclusion}
\label{sec:con}

In this paper, we explored the split thickness of graphs. 
We proved tight bounds on the planar split thickness of complete and complete bipartite graphs. We proved that recognizing $2$-splittable graphs is NP-hard, and remains NP-hard even when the maximum degree is restricted to 15. A natural direction would be to examine the complexity for graphs with small maximum degree. 

We also proved that the planar split thickness of a graph is approximable 
within a constant factor. Furthermore, if the treewidth of the input graph is 
bounded, then for any fixed~$k$, one can decide $k$-splittability into planar 
graphs in linear time. However, this algorithm makes ample use of
Courcelle's theorem, so a more practical algorithm would be desirable.

We also showed that any graph that can be
embedded on the torus or projective plane is 2-splittable.
It remains open whether graphs with genus~$k$ are~$(k+1)$-splittable.

\subsection*{Acknowledgments} 
Most of the results of this paper were obtained at the McGill-INRIA-Victoria Workshop on Computational Geometry, Barbados, February 2015. We would like to thank the organizers of these events, as well as many participants for fruitful discussions and suggestions. The first, fourth, sixth and eighth authors acknowledge the support from NSF grant 1228639, 2012C4E3KT PRIN Italian National Research Project, PEPS egalite project and NSERC, respectively.
  
\bibliographystyle{splncs03}
\bibliography{abbrv,references}

\end{document}